\documentclass[12pt,english]{article}
\usepackage[T1]{fontenc}
\usepackage[latin9]{inputenc}
\usepackage{geometry}
\usepackage{amsmath}
\usepackage{amsthm}
\usepackage{amssymb}
\usepackage{xcolor}
\usepackage{comment}

\makeatletter

\theoremstyle{plain}
\newtheorem{thm}{\protect\theoremname}
\theoremstyle{plain}
\newtheorem{lem}{\protect\lemmaname}
\theoremstyle{remark}
\newtheorem{rem}{\protect\remarkname}
\newtheorem{con}{\protect\conjecturename}
\ifx\proof\undefined\
  \newenvironment{proof}[1][\proofname]{\par
    \normalfont\topsep6\p@\@plus6\p@\relax
    \trivlist
    \itemindent\parindent
    \item[\hskip\labelsep
          \scshape
      #1]\ignorespaces
  }{%
    \endtrivlist\@endpefalse
  }
  \providecommand{\proofname}{Proof}
\fi

\newcommand{\R}{\mathbb{R}}

\date{}

\makeatother

\usepackage{babel}
\providecommand{\lemmaname}{Lemma}
\providecommand{\theoremname}{Theorem}
\providecommand{\conjecturename}{Conjecture}
\providecommand{\remarkname}{Remark}

\begin{document}
\title{

Stochastic domination of exit times for random walks and Brownian motion with drift 
}


	\author{Xi Geng \footnote{Email : xi.geng@unimelb.edu.au}\\ School of Mathematics and Statistics\\
		University of Melbourne\\ \vspace{0.3cm} Parkville, VIC Australia\\
		Greg Markowsky \footnote{E-mail :
			greg.markowsky@monash.edu} \\
			School of Mathematics \\
			Monash University \\
			Clayton, VIC Australia}

\maketitle
\begin{abstract}
In this note, by an elementary use of Girsanov's transform we show that the exit time for either a biased random walk or a drifted Brownian motion on a symmetric interval is stochastically monotone with respect to the drift parameter. In the random walk case, this gives an alternative proof of a recent result of E. Pek\"oz and R. Righter in 2024, while the Brownian motion case is the continuous analogue as discussed in the same paper. Our arguments in both discrete and continuous cases are parallel to each other. We also outline a simple SDE proof for the Brownian case based on a standard comparison theorem.

\end{abstract}

\section{Introduction} \label{intro}
In a series of elegant recent papers, several variants of the following problem were addressed: determine the optimal drift for 1-dimensional simple random walk or Brownian motion starting at $0$ to stay for as long as possible in a symmetric interval. To be precise, let $\{S_{n}^{p}:n\geqslant0\}$
denote the simple random walk starting at the origin with one-step
distribution 
\[
\mathbb{P}(X=1)=p,\ \mathbb{P}(X=-1)=1-p.
\]
This is often referred to as the {\it biased random walk}. 
Given $k\in\mathbb{Z}_{+},$ define 

$$\sigma_{p}^{k} \triangleq\inf\{n:S_{n}^{p}=\pm k\}.$$ 
In \cite[Lemma 2]{zhang2023finding} it was shown that the function $p \to E[\sigma_{p}^{k}]$ is increasing for $p \in (0,1/2)$ and decreasing for $p \in (1/2,1)$, thereby showing that this function is maximized at $p=1/2$. Later, in \cite{pekoz2024increasing, pekoz2022fair}, it was shown that the stopping time $\sigma_{p}^{k}$ is stochastically maximized at $p=1/2$, and furthermore if $1/2 \leqslant p_1 < p_2 \leqslant 1$ then $\sigma_{p_1}^{k}$ stochastically dominates $\sigma_{p_2}^{k}$ (Theorem \ref{thm:Main} (i) below). Their proof is based on the use of a clever coupling argument. It was also asserted in \cite{pekoz2024increasing} that an appeal to Donsker's theorem allows one to deduce the analogous result for Brownian motion (Theorem \ref{thm:Main} (ii) below) as a corollary to the result for random walk.

The proofs given in \cite{pekoz2024increasing,pekoz2022fair,zhang2023finding} are ingenious but combinatorial in nature, and do not seem easy to be modified in order to prove Theorem \ref{thm:Main} (ii) directly (i.e. without proving the result first for simple random walk and then invoking Donsker's theorem). It is natural then to search for a direct proof in the Brownian case. In the end we were able to find two such proofs, one by performing a Girsanov change of measure, and the other by translating the problem into a question about SDE's and applying a standard comparison theorem. Both proofs are relatively short and non-technical. The purpose of this note is to present these two proofs.\footnote{In private communication, we were informed by Stephen Muirhead that there was yet a third proof (for the Brownian motion case) based on Anderson's inequality for Gaussian measures (cf. \cite[Corollary 3.5]{AK18}), and using such approach the result could further be generalised to drifted Gaussian processes.}

The following is the random walk result proved in \cite{pekoz2024increasing, pekoz2022fair}, as well as the analogous Brownian motion result that we will focus on.

\begin{thm}
\label{thm:Main}(i) {[}Random walk{]} 
With notation as above, for any fixed $n\in\mathbb{N}$ the function 
\[
(0,1)\ni p\mapsto\mathbb{P}(\sigma_{p}^{k}>n)
\]
is increasing on $(0,1/2]$ and decreasing on $[1/2,1)$. \\
(ii) {[}Brownian motion{]} Let $B_{t}$ denote
the standard one-dimensional Brownian motion. Given $\lambda\in\mathbb{R}$
and $b>0,$ define 
\[
\tau_{\lambda}^{b}\triangleq\inf\{t:B_{t}+\lambda t\notin(-b,b)\}.
\]
Then for any fixed $t>0,$ the function 
\[
\lambda\mapsto\mathbb{P}(\tau_{\lambda}^{b}>t)
\]
is decreasing in $\lambda$ on $[0,\infty)$. 
\end{thm}
%

\begin{rem}
Theorem \ref{thm:Main} seems surprisingly resistant to generalisations. The conclusion is in general not true if the drift is assumed to be non-constant (a simple comparison relation between the drifts does not have a clear implication on the relation between exit times). For instance, if the drift is allowed to depend on the position of the path, one could make a large positive drift when the process goes negative which would push it back to the center, resulting in longer stay in the interval. The theorem also fails to hold in general if the drift is assumed to be time-dependent. 
\end{rem}

\if2
The Brownian motion case, which is to some extent easier to understand,
is based on a simple application of Girsanov's theorem together with
the crucial observation that the exit location and the exit time are
independent even in the drifted case (Reuter's theorem). The discrete case essentially uses a discrete form of Girsanov's theorem (Lemma \ref{lem:CoM} (ii)), which may be of pedagogical interest since it is easy to check "by hand", while the standard Girsanov's theorem for Brownian motion involves an understanding of subtleties of quadratic variation and Levy's characterization of Brownian motion (see \cite[Thm. 10.15]{klebaner2012introduction}). 
\fi 

In the next two sections, we develop the Girsanov proof of the theorem in both the Brownian motion and random walk contexts. Our arguments are entirely parallel in these two cases. In Section \ref{sec:SDEPf}, we outline the SDE proof for the Brownian case based on a well-known comparison theorem. This SDE proof is the continuous-time analogue of an argument which appears for the discrete-time case in \cite{ross2023second}.

Incidentally, as mentioned earlier, part $(ii)$ of the theorem was stated as a corollary to part $(i)$ in \cite{pekoz2024increasing} by invoking the weak convergence of simple random walk to Brownian motion. While this is certainly correct in spirit, it is the opinion of the authors of this paper that making such an argument rigorous would present various unpleasant technicalities. It may be beneficial and simpler to just have a direct proof in the Brownian case, as we have done here.


\section{The Girsanov proof}

In this section, we present the change-of-measure proof for Brownian motion. As indicated earlier, this proof can be directly adapted, mutatis mutandis, to prove the discrete result as well, and we will discuss this in the next section.

Let us begin by setting up the problem on the canonical sample space. We take $\Omega$ to be the continuous path
space, $B_{t}(\omega)\triangleq\omega_{t}$ to be the coordinate process
and ${\cal F}_{t}$ to be the natural filtration of $B_{t}$.

Let
$\mathbb{Q}_{\lambda}$ denote the probability measure
over $\Omega$ under which $B_{t}$ becomes a Brownian motion with
drift $\lambda t$. Note that $\mathbb{Q}_{0}$ is simply the standard Wiener measure. We also set $\tau(\omega)\triangleq\inf\{t:B_{t}(\omega)\notin(-b,b)\}$
($b>0$ is given fixed). The following version of Girsanov's theorem is needed for our purpose.

\begin{lem} \label{girs}
For each
$t>0$, one has 
\[
\frac{d\mathbb{Q}_{\lambda}}{d\mathbb{Q}_{0}}=e^{\lambda B_t- \frac{1}{2}\lambda^2 t}\ \ \ \text{on }{\cal F}_{t} \quad and \quad \frac{d\mathbb{Q}_{\lambda}}{d\mathbb{Q}_{0}}=e^{\lambda B_\tau- \frac{1}{2}\lambda^2 \tau}\ \ \ \text{on }{\cal F}_{\tau}.
\]
Furthermore, given two drift values $\lambda_1, \lambda_2$ we have 
\[
\frac{d\mathbb{Q}_{\lambda_2}}{d\mathbb{Q}_{\lambda_1}}=e^{-(\lambda_{1}-\lambda_{2})B_{t}+\frac{\lambda_{1}^{2}-\lambda_{2}^{2}}{2}t}\ \ \ \text{on }{\cal F}_{t} \quad and \quad \frac{d\mathbb{Q}_{\lambda_2}}{d\mathbb{Q}_{\lambda_1}}=e^{-(\lambda_{1}-\lambda_{2})B_{\tau}+\frac{\lambda_{1}^{2}-\lambda_{2}^{2}}{2}\tau}\ \ \ \text{on }{\cal F}_{\tau}.
\]

\end{lem}
\begin{proof}
The general Girsanov's theorem is contained in e.g. \cite[Thm. 10.15]{klebaner2012introduction} or \cite[Chap. VIII, Proposition 1.3]{revuz1991continuous}. Here the only needed observation is that if we set $M_t \triangleq e^{\lambda B_t- \frac{1}{2}\lambda^2 t}$, then the stopped martingale $\{M_{\tau \wedge t}:n \geqslant 0\}$ is bounded and thus uniformly integrable. 

\end{proof}
The following lemma appears in \cite[P.84]{RW20}, so we omit its proof, though we note that the proof is identical, line for line, with that of Lemma \ref{lem:Indep2} below, provided one replaces $S$ with $B$, $\sigma$ with $\tau$, $\mathbb{Q}_p$ with $\mathbb{Q}_{\lambda}$, and the discrete-time martingale $M_n$ from Lemma \ref{girs_disc} with the the exponential martingale $e^{\lambda B_t-\frac{1}{2}\lambda^2 t}$ of Brownian motion.
\begin{lem}
\label{lem:Indep} Under $\mathbb{Q}_{\lambda},$
the random variables $B_{\tau}$ and $\tau$ are independent. 
\end{lem}

We also need the following elementary estimate.
\begin{lem}
\label{lem:EleLem}Let $X$ be a random variable and let $M$ be any real number. Suppose that $P(X \leqslant M) > 0$. Then $\mathbb{E}[X|X\leqslant M]\leqslant\mathbb{E}[X]$.
\end{lem}
\begin{proof}

If we let $g(x) = 1_{x \leq M}$, then $g(x)$ is a decreasing function, so that $(X-a)(g(X)-g(a)) \leq 0$ a.s. for any $a \in \R$. Thus, $0 \geq E[(X-E[X])(g(X)-g(E[X]))] = E[Xg(X)] - E[X]E[g(X)]$, so that $X$ and $1_{X \leq M}$ are non-positively  correlated. Thus, 
 $\mathbb{E}[X|X\leqslant M] = \frac{E[X1_{X \leqslant M}]}{P(X \leqslant M)} \leqslant\frac{E[X]E[1_{X \leqslant M}]}{P(X \leqslant M)} = E[X]$.

\end{proof}
\begin{proof}[Proof of Theorem \ref{thm:Main} (ii)] 

Let $0\leqslant\lambda_{1}<\lambda_{2}$ be given fixed.
The starting observation is that 
\[
\mathbb{Q}_{\lambda_{2}}(\tau>t)=\mathbb{E}_{\lambda_{2}}\big[{\bf 1}_{\{\tau>t\}}\big]=\mathbb{E}_{\lambda_{1}}\big[e^{-(\lambda_{1}-\lambda_{2})B_{\tau}+\frac{\lambda_{1}^{2}-\lambda_{2}^{2}}{2}\tau}{\bf 1}_{\{\tau>t\}}\big],
\]
by Lemma \ref{girs}. The independence
property given by Lemma \ref{lem:Indep} implies that 
\begin{align*}
 & \mathbb{E}_{\lambda_{1}}\big[e^{-(\lambda_{1}-\lambda_{2})B_{\tau}+\frac{\lambda_{1}^{2}-\lambda_{2}^{2}}{2}\tau}{\bf 1}_{\{\tau>t\}}\big]\\
 & =\mathbb{E}_{\lambda_{1}}\big[e^{-(\lambda_{1}-\lambda_{2})B_{\tau}}\big]\mathbb{E}_{\lambda_{1}}\big[e^{\frac{\lambda_{1}^{2}-\lambda_{2}^{2}}{2}\tau}{\bf 1}_{\{\tau>t\}}\big]=\frac{\mathbb{E}_{\lambda_{1}}\big[e^{\frac{\lambda_{1}^{2}-\lambda_{2}^{2}}{2}\tau}{\bf 1}_{\{\tau>t\}}\big]}{\mathbb{E}_{\lambda_{1}}\big[e^{\frac{\lambda_{1}^{2}-\lambda_{2}^{2}}{2}\tau}\big]} \\
 & = \frac{\mathbb{E}_{\lambda_{1}}\big[e^{\frac{\lambda_{1}^{2}-\lambda_{2}^{2}}{2}\tau}\big|\tau>t\big]}{\mathbb{E}_{\lambda_{1}}\big[e^{\frac{\lambda_{1}^{2}-\lambda_{2}^{2}}{2}\tau}\big]} \mathbb{Q}_{\lambda_{1}}(\tau>t)
\end{align*}
It follows that 
\[
\mathbb{Q}_{\lambda_{2}}(\tau>t)\leqslant\mathbb{Q}_{\lambda_{1}}(\tau>t)\iff\mathbb{E}_{\lambda_{1}}\big[e^{\frac{\lambda_{1}^{2}-\lambda_{2}^{2}}{2}\tau}\big|\tau>t\big]\leqslant\mathbb{E}_{\lambda_{1}}\big[e^{\frac{\lambda_{1}^{2}-\lambda_{2}^{2}}{2}\tau}\big].
\]
The latter inequality is a direct consequence of Lemma \ref{lem:EleLem}
since $\lambda_{1}^{2}<\lambda_{2}^{2}$. 
\end{proof}

\section{The discrete case} \label{discrete section}

We now discuss the proof of Theorem \ref{thm:Main} $(i)$ (the discrete case). Let $S_n$ denote unbiased simple random walk with respect to a measure $\mathbb{Q}_{1/2}$; that is, the simple random walk with equal probability of moving to the right and to the left. We need to be able to change measure in order to transform one random walk into another with a different bias. This requires a discrete analog of Lemma \ref{girs}, i.e. a discrete-time Girsanov's Theorem, and for this we need the discrete analog of the exponential martingale $e^{\lambda B_t- \frac{1}{2}\lambda^2 t}$. This is contained in the following lemma.

\begin{lem} \label{girs_disc}
(i) For $p,q \in (0,1)$ with $p+q=1$, the process $M_n = (2\sqrt{pq})^n \Big(\sqrt{\frac{p}{q}}\Big)^{S_n}$ is a martingale with respect to $\mathbb{Q}_{1/2}$.

\vspace{2mm}\noindent (ii) If we define a new measure $\mathbb{Q}_p$ by

\[
\frac{d\mathbb{Q}_{p}}{d\mathbb{Q}_{1/2}}= (2\sqrt{pq})^n \Big(\sqrt{\frac{p}{q}}\Big)^{S_n} \text{on }{\cal F}_{n},
\]
then with respect to $\mathbb{Q}_p$ the process $S_n$ is biased random walk. To be precise, $\mathbb{Q}_p(S_{n+1} = r | S_n = r-1) = p$ and $\mathbb{Q}_p(S_{n+1} = r | S_n = r+1) = q$. Furthermore,

\[
\frac{d\mathbb{Q}_{p}}{d\mathbb{Q}_{1/2}}= (2\sqrt{pq})^\sigma \Big(\sqrt{\frac{p}{q}}\Big)^{S_\sigma} \text{on }{\cal F}_{\sigma}.
\]
(iii) Given two biases $p_1, p_2$, we have

\[
\frac{d\mathbb{Q}_{p_2}}{d\mathbb{Q}_{p_1}}= \Big(\sqrt{\frac{p_2q_2}{p_1q_1}}\Big)^n \Big(\sqrt{\frac{p_2q_1}{p_1q_2}}\Big)^{S_n} \ \ \ \text{on }{\cal F}_{n} \quad and \quad \frac{d\mathbb{Q}_{p_2}}{d\mathbb{Q}_{p_1}}= \Big(\sqrt{\frac{p_2q_2}{p_1q_1}}\Big)^\sigma \Big(\sqrt{\frac{p_2q_1}{p_1q_2}}\Big)^{S_\sigma}\ \ \ \text{on }{\cal F}_{\sigma}.
\]

\end{lem}

\begin{proof}
For part (i), we note that $S_n$ can be taken as $S_n = \sum_{j=1}^n X_j$, where the $X_j$'s are independent random variables equal to $\pm 1$ with equal probabilities. We note that 

$$
\mathbb{E}_{\mathbb{Q}_{1/2}}[\Big(\sqrt{\frac{q}{p}}\Big)^{X_j}] = \frac{1}{2}(\sqrt{\frac{q}{p}} + \sqrt{\frac{p}{q}}) = \frac{1}{2\sqrt{pq}}.
$$
The process $M_n$ is therefore revealed to be a product martingale with respect to $\mathbb{Q}_{1/2}$.

For part (ii), we note first that part (i) of this lemma was required to show that $\frac{d\mathbb{Q}_{p}}{d\mathbb{Q}_{1/2}}$ defines a genuine change of measure. We may then calculate as follows (for any pair $n,r$ such that $\mathbb{Q}_p(S_n = r-1)>0$):

\begin{align*}
    \mathbb{Q}_p & (S_{n+1} = r | S_n = r-1) = \frac{\mathbb{E}_{\mathbb{Q}_p}[1_{\{S_{n+1} = r\}} 1_{\{S_n = r-1\}}]}{\mathbb{E}_{\mathbb{Q}_p}[1_{\{S_{n} = r-1\}}]} \\
    & = \frac{\mathbb{E}_{\mathbb{Q}_{1/2}}\Big[(2\sqrt{pq})^{n+1} \Big(\sqrt{\frac{p}{q}}\Big)^{S_{n+1}}1_{\{S_{n+1} = r\}} 1_{\{S_n = r-1\}}\Big]}{\mathbb{E}_{\mathbb{Q}_{1/2}}\Big[(2\sqrt{pq})^n \Big(\sqrt{\frac{p}{q}}\Big)^{S_n} 1_{\{S_{n} = r-1\}}\Big]} \\
    &= \frac{(2\sqrt{pq})^{n+1} \Big(\sqrt{\frac{p}{q}}\Big)^{r}\mathbb{E}_{\mathbb{Q}_{1/2}}\Big[1_{\{S_{n+1} = r\}} 1_{\{S_n = r-1\}}\Big]}{(2\sqrt{pq})^{n} \Big(\sqrt{\frac{p}{q}}\Big)^{r-1}\mathbb{E}_{\mathbb{Q}_{1/2}}\Big[1_{\{S_{n} = r-1\}}\Big]} \\
    & = 2p \mathbb{Q}_{1/2}(S_{n+1} = r | S_n = r-1) = p.
\end{align*}
A similar calculation shows $\mathbb{Q}_p(S_{n+1} = r | S_n = r+1)=q$. The second part of part (ii) holds because $M_n$ is uniformly integrable.

As in the proof of Lemma \ref{girs}, part (iii) follows from (ii) by the chain rule.
\end{proof}

\begin{lem}
\label{lem:Indep2} Under $\mathbb{Q}_{p},$
the random variables $S_{\sigma}$ and $\sigma$ are independent. 
\end{lem}

\begin{proof}
The result is obvious if $p=1/2$. For the biased
case, to ease notation we write $z\triangleq\sqrt{p/q}$ and $w\triangleq(\sqrt{p/q}+\sqrt{q/p})/2.$
Let $f,g$ be arbitrary test functions. According to Lemma \ref{girs_disc},
one has 
\begin{align*}
 & \mathbb{E}^{\mathbb{Q}_{p}}\big[f(S_{\sigma})g(\sigma)\big]\\
 & =\mathbb{E}^{\mathbb{Q}_{1/2}}\big[z^{S_{\sigma}}w^{-\sigma}f(S_{\sigma})g(\sigma)\big]=\mathbb{E}^{\mathbb{Q}_{1/2}}\big[z^{S_{\sigma}}f(S_{\sigma})\big]\mathbb{E}^{\mathbb{Q}_{1/2}}\big[w^{-\sigma}g(\sigma)\big]\\
 & =\mathbb{E}^{\mathbb{Q}_{1/2}}\big[z^{S_{\sigma}}f(S_{\sigma})\big]\mathbb{E}^{\mathbb{Q}_{1/2}}\big[w^{-\sigma}g(\sigma)\big]\times\mathbb{E}^{\mathbb{Q}_{1/2}}\big[z^{S_{\sigma}}w^{-\sigma}\big]\ \ \ (\text{the last quantity is just }1)\\
 & =\mathbb{E}^{\mathbb{Q}_{1/2}}\big[z^{S_{\sigma}}w^{-\sigma}f(S_{\sigma})\big]\mathbb{E}^{\mathbb{Q}_{1/2}}\big[z^{S_{\sigma}}w^{-\sigma}g(\sigma)\big]=\mathbb{E}^{\mathbb{Q}_{p}}\big[f(S_{\sigma})\big]\mathbb{E}^{\mathbb{Q}_{p}}\big[g(\sigma)\big].
\end{align*}
The result thus follows. 
\end{proof}

\begin{proof}[Proof of Theorem \ref{thm:Main} (i)] We could argue here that the proof of this result is the same as for part $(ii)$ of this theorem with appropriate substitutions, and with Lemmas \ref{girs_disc} and \ref{lem:Indep2} in place of Lemmas \ref{girs} and \ref{lem:Indep} (Lemma \ref{lem:EleLem} needs no modification). However, we include the proof here, for completeness.

Let $1/2\leqslant p_{1}<p_{2}$ be given fixed.
We note that 
\[
\mathbb{Q}_{p_{2}}(\sigma>n)=\mathbb{E}_{p_{2}}\big[{\bf 1}_{\{\sigma>n\}}\big]=\mathbb{E}_{p_{1}}\big[\Big(\sqrt{\frac{p_2q_2}{p_1q_1}}\Big)^\sigma \Big(\sqrt{\frac{p_2q_1}{p_1q_2}}\Big)^{S_\sigma}{\bf 1}_{\{\sigma>n\}}\big],
\]
by Lemma \ref{girs_disc}. The independence
property given by Lemma \ref{lem:Indep2} implies that 
\begin{align*}
 & \mathbb{E}_{p_{1}}\big[\Big(\sqrt{\frac{p_2q_2}{p_1q_1}}\Big)^\sigma \Big(\sqrt{\frac{p_2q_1}{p_1q_2}}\Big)^{S_\sigma}{\bf 1}_{\{\sigma>n\}}\big]\\
 & =\mathbb{E}_{p_{1}}\big[\Big(\sqrt{\frac{p_2q_1}{p_1q_2}}\Big)^{S_\sigma}\big]\mathbb{E}_{p_{1}}\big[\Big(\sqrt{\frac{p_2q_2}{p_1q_1}}\Big)^\sigma{\bf 1}_{\{\sigma>n\}}\big]=\frac{\mathbb{E}_{p_{1}}\big[\Big(\sqrt{\frac{p_2q_2}{p_1q_1}}\Big)^\sigma{\bf 1}_{\{\sigma>n\}}\big]}{\mathbb{E}_{p_{1}}\big[\Big(\sqrt{\frac{p_2q_2}{p_1q_1}}\Big)^\sigma\big]} \\
 & = \frac{\mathbb{E}_{p_{1}}\big[\Big(\sqrt{\frac{p_2q_2}{p_1q_1}}\Big)^\sigma\big|\sigma>n\big]}{\mathbb{E}_{p_{1}}\big[\Big(\sqrt{\frac{p_2q_2}{p_1q_1}}\Big)^\sigma \big]} \mathbb{Q}_{p_{1}}(\sigma>n)
\end{align*}
It follows that 
\[
\mathbb{Q}_{p_{2}}(\sigma>n)\leqslant\mathbb{Q}_{p_{1}}(\sigma>n)\iff\mathbb{E}_{p_{1}}\big[\Big(\sqrt{\frac{p_2q_2}{p_1q_1}}\Big)^\sigma\big|\sigma>n\big]\leqslant\mathbb{E}_{p_{1}}\big[\Big(\sqrt{\frac{p_2q_2}{p_1q_1}}\Big)^\sigma\big].
\]
Note that $\frac{p_2q_2}{p_1q_1} < 1$, since $p_2 > p_1$. Lemma \ref{lem:EleLem}
 therefore completes the proof. 

 \end{proof}

\section{The SDE proof}\label{sec:SDEPf}

In this section, we outline the SDE proof of Theorem \ref{thm:Main} (ii). The reader should be aware that this is just a sketch of the argument, and that, for brevity, we do not give the technical details here.

A natural idea is to try to represent the modulus of $B^\lambda_t \triangleq B_{t}+\lambda t$ as an \textit{It\^o process}, namely the solution to an SDE of the form

\begin{equation} \label{goodSDE}
    dX_t=\sigma(X_t)dB_t+b(X_t)dt.
\end{equation} The reason for doing this is that it would allow us to invoke a well-known comparison theorem due to Ikeda and Watanabe \cite{ikeda1977comparison}. This theorem states essentially that if $X^1$ and $X^2$ are diffusions on the same probability space, starting at the same value and each satisfying equations of the form (\ref{goodSDE}) with the same $\sigma$ but with different drift terms $b_1,b_2$ with $b_1(x) \leqslant b_2(x)$ for all $x$, then $X^1_t \leqslant X^2_t$ for all $t$ a.s. (naturally there are conditions on $\sigma$ and the $b$'s, but we do not worry about them for the moment). This would allow us to deduce the desired result, as the exit time from the symmetric interval in question is simply the hitting time of the corresponding level by the process $|B^\lambda_t|$.

We must first remark, however, that the process $|B_{t}^{\lambda}|$ does not satisfy
an SDE of the form (\ref{goodSDE}). This is because the semimartingale
decomposition of $|B_{t}^{\lambda}|$ involves its local time at zero,
which is not absolutely continuous with respect to $dt$ (when $\lambda=0$,
Tanaka's formula gives $|B_{t}|=dW_{t}+dL_{t}$ where $W_{t}$ is
a Brownian motion and $L_{t}$ is the local time of $B$ at zero). 

On the other hand, the process $Y_t \triangleq |B_{t}^{\lambda}|^{2}=(B_{t}^{\lambda})^{2}$
does satisfy an SDE of the required form, which we shall derive below. The case when $\lambda=0$
is straightforward; here $Y_t = B_{t}^{2}$, and one has 
\begin{equation} 
dY_{t}=2B_{t}dB_{t}+dt=2\sqrt{Y_{t}}\cdot\frac{B_{t}}{\sqrt{Y_{t}}}dB_{t}+dt=2\sqrt{Y_{t}}dW_{t}+dt,\label{eq:SDENoDrif}
\end{equation}
where $W_{t}\triangleq\int_{0}^{t}\frac{B_{s}}{\sqrt{Y_{s}}}dB_{s}$
is a Brownian motion by L\'evy's characterisation. The case $\lambda\neq0$ requires
extra care since the same calculation does not lead to a meaningful
SDE for $Y_{t}$ (the drift part contains $B_{t}^{\lambda}$ which
cannot be expressed in terms of $Y_{t}$). 

We take a Markovian perspective to write down the intrinsic SDE for
$Y_{t}.$ It is well known that the process $|B_{t}|$ can be equivalently
viewed as a Markov process with generator 
\[
{\cal A}f=\frac{1}{2}f'';\ {\cal D}({\cal A})=\big\{ f\in C_{b}^{2}([0,\infty)):f'(0+)=0\big\}.
\]
To derive the generator for the process $X_{t}\triangleq|B_{t}^{\lambda}|$,
let $x>0$ be given fixed (it represents the current state $X_{t}=x$).
Explicit calculation shows that 
\[
\mathbb{P}(B_{t}^{\lambda}=x|X_{t}=x)=\frac{e^{\lambda x}}{e^{\lambda x}+e^{-\lambda x}},\ \mathbb{P}(B_{t}^{\lambda}=-x|X_{t}=x)=\frac{e^{-\lambda x}}{e^{\lambda x}+e^{-\lambda x}}.
\]
Given $X_{t}=x,$ if $B_{t}^{\lambda}=x$ the process evolves like
$B_{s}+\lambda s$, while if $B_{t}^{\lambda}=-x$ the process evolves
like $-B_{s}-\lambda s$ ($s\in[t,t+\delta t]$). Since the Brownian
motion $B$ is symmetric, it is clear that the diffusive part (the
second order term) of the generator of $X_{t}$ is also $\frac{1}{2}\frac{d^{2}}{dx^{2}}$.
Its drift part (the first order term) is given by 

\[
\mathbb{P}(B_{t}^{\lambda}=x|X_{t}=x)\times\lambda+\mathbb{P}(B_{t}^{\lambda}=-x|X_{t}=x)\times(-\lambda)=\lambda\tanh\lambda x.
\]
In other words, $X_{t}$ is a Markov process with generator
\[
{\cal L}^{X}=\frac{1}{2}\frac{d^{2}}{dx^{2}}+\lambda\tanh(\lambda x)\frac{d}{dx};\ {\cal D}({\cal L}^{X})=\big\{ f\in C_{b}^{2}([0,\infty)):f'(0+)=0\big\}.
\]
A simple change of variables $x=\sqrt{y}$ shows that the generator
of $Y_{t}$ is given by 
\[
{\cal L}^{Y}=2y\frac{d^{2}}{dy^{2}}+\big(1+2\lambda\sqrt{y}\tanh(\lambda\sqrt{y})\big)\frac{d}{dy},
\]
and the corresponding SDE for $Y_{t}$ is 
\begin{equation}
dY_{t}=2\sqrt{Y_{t}}dW_{t}+(1+2\lambda\sqrt{Y_{t}}\tanh(\lambda\sqrt{Y_{t}}))dt.\label{eq:SDEDrif}
\end{equation}

Note that this agrees with (\ref{eq:SDENoDrif}) when $\lambda=0$. We now observe that the function
\[
\lambda \to 1+2\lambda\sqrt{y}\tanh(\lambda\sqrt{y})
\]
is increasing for $\lambda \in [0,\infty)$ for all $y$,
and therefore the comparison theorem of Ikeda-Watanabe easily gives our desired result. \qed

\vspace{12pt}

\begin{rem} Invoking the theorem of Ikeda-Watanabe requires checking certain conditions on the functions comprising the SDE (\ref{eq:SDEDrif}). These conditions are slightly technical to state but straightforward to verify for the SDE in question, so we have chosen not to include them. 

Unlike $X_t$ itself, the process $Y_t$ is indeed an It\^o diffusion (it does satisfy the SDE (\ref{eq:SDEDrif})) because the local time term will not appear. In fact, in the decomposition
\begin{equation}
dY_{t}=2X_{t}dX_{t}+dX_{t}\cdot dX_{t},\label{eq:YSDE}
\end{equation}
the local time term $X_{t}dL_{t}$ (coming from the first term in
(\ref{eq:YSDE})) vanishes identically due to the fact that $L_{t}$
only increases when $X_{t}=0$ (this is a basic property of the local
time). 

From general theory, the SDE (\ref{eq:SDEDrif}) has a unique strong
solution for every initial condition $Y_{0}=y\geqslant0.$ One can
then show that its solution (with $Y_{0}=0$) has the same distribution
as the process $|B_{t}+\lambda t|^{2}$. This can be taken as a more
direct approach to justify the above considerations. 
\end{rem}

\begin{rem}
    An analogous argument was applied to the modulus of the biased random walk $|S_n|$ in \cite[Section 5.8]{ross2023second}, in order to show that $|S_n|$ is a Markov chain and to compute the transition probabilities. Part (i) of Theorem \ref{thm:Main} is then deduced as a consequence, much in the same manner that we have concluded part (ii) from the SDE above.
\end{rem}

\section{Acknowledgements}

XG gratefully acknowledges the ARC support DE210101352. Both authors are thankful to Nathan Ross and Stephen Muirhead for valuable discussions and suggestions on the problem. They would also like to extend their thanks to an anonymous referee and an anonymous editor for suggestions which have significantly improved and simplified the paper.

\bibliographystyle{plainurl}
\bibliography{citation}

\end{document}